\documentclass[11pt]{amsart}
\usepackage{adjustbox, amssymb, amsfonts, amsmath, bbm, bm, mathrsfs, mathtools, pgfplots, physics, stackrel, tikz, tikz-cd,  todonotes,  verbatim, xypic}
\usetikzlibrary{intersections, calc, matrix, decorations.pathreplacing}
\usepackage[shortlabels]{enumitem}
\usepackage[bbgreekl]{mathbbol}
\DeclareMathSymbol\bbDelta  \mathord{bbold}{"01}

\definecolor{internalremarkcolor}{rgb}{1.0, 0.71, 0.76}

\usepackage[outline]{contour} \contourlength{2pt}
\usepackage{hyperref}
\hypersetup{
    colorlinks=false, 
    linktoc=all,     
    linkcolor=blue,  
}

\pgfplotsset{compat=newest}

\usepackage{geometry}
\allowdisplaybreaks

\theoremstyle{definition}
\newtheorem{theorem}{Theorem}[section]
\newtheorem{lemma}[theorem]{Lemma}

\newtheorem{proposition}[theorem]{Proposition}
\newtheorem{definition}[theorem]{Definition}

\numberwithin{equation}{section}
\numberwithin{figure}{section}

\newcommand{\Fin}{\text{Fin}}
\newcommand{\WFin}{\Omega\text{Fin}}
\newcommand{\catWFin}{\mathbf{\Omega}\textbf{Fin}}
\newcommand{\sSet}{\textbf{sSet}}

\newcommand\subsetsim{\mathrel{%
  \ooalign{\raise0.2ex\hbox{$\subset$}\cr\hidewidth\raise-0.8ex\hbox{\scalebox{0.9}{$\sim$}}\hidewidth\cr}}}

\newcommand{\del}{ \partial}

\newcommand{\Cech}{\check{C}}

\newcommand{\Ch}{\mathbf{Ch}_{\bullet}}

\DeclareFontFamily{U}{dmjhira}{}
\DeclareFontShape{U}{dmjhira}{m}{n}{
  <-> s*[0.95] dmjhira
}{}
\DeclareFontSubstitution{U}{dmjhira}{m}{n}

\DeclareFontFamily{OT1}{pzc}{}
\DeclareFontShape{OT1}{pzc}{m}{it}{<-> s * [1.10] pzcmi7t}{}
\DeclareMathAlphabet{\mathpzc}{OT1}{pzc}{m}{it}

\newcommand\restr[2]{{
  \left.\kern-\nulldelimiterspace 
  #1 
  \vphantom{\big|} 
  \right|_{#2} 
  }}

\author[C. Glass]{Cheyne Glass}
\address{Cheyne Glass, State University of New York at New Paltz, Department of Mathematics, 1 Hawk Dr., New Paltz, NY 12561}
  \email{glassc@newpaltz.edu}

\title[Homotopy-theoretic least squares regression]{Homotopy-theoretic least squares regression}

\begin{document}
\maketitle

\begin{abstract}
A presheaf of complexes is constructed on a category of weighted finite subsets of a fixed Euclidean space. To each object, a Koszul complex is assigned which resolves the coordinate ring of least squares solutions on that data set for a choice of particular model (ie ``y=mx+b''). In order to obtain a total \v{C}ech-theoretic complex where the $0$-cocycles resemble locally defined least squares solutions gluing together up to homotopy, the coefficient rings for the Koszul complexes over each subset are linearized near a least squares solution. While these new linearized complexes do not immediately assemble into a presheaf, additional change-of-coordinates maps restore functoriality. Evaluating this new presheaf of complexes on a cover, its total-degree-0-cocycles of this \v{C}ech-Koszul bicomplex reveals (higher) homotopies between the discrepancies of least squares solutions on (higher) overlaps. A toy example with 5 data points is worked out in full elementary detail. 
\end{abstract}

\tableofcontents
\section*{Introduction}
The purpose of this short note is to offer some constructions from algebraic topology/geometry which I hope can find some traction in the world of applied sheaf theory. Sheaf theory in mathematics, especially in questions related to the physical/analog world, has been implicitly baked-in to many constructions over the last century. The philosophy is fairly reasonable: instead of looking for global solutions, just find local solutions which ``glue'' together appropriately. From there, the ideas of ``glueing up to homotopy'' naturally arise and presentations of these ``infinity sheaves'' take the form of $dg$-presheaves, presheaves of complexes, simplicial presheaves, etc. It is surprising that a theory of ``regression analysis up to homotopy'' has seemingly not been explored, given all of the theoretical tools were sitting there. Presumably just as in the physical world, such a theory would likely add to accuracy in predictions. There is no claim that in this note the reader will find a perfected, ready-to-implement algorithm which will change the world of regression analysis. Instead, the idea is simply to offer a potential path forward for those with more expertise in applied settings to use these infinity sheaf tools.  

In \cite{GZ} and \cite{GTZ} we took many of the constructions from \cite{TT} and pushed them into statements on the level of infinity stacks. In those particular works, \v{C}ech-theoretic cocycles are constructed by locally resolving the diagonal in $\mathbb{C}^n \times \mathbb{C}^n$ using Koszul complexes in a fairly natural way, and then gluing the resolutions together with chain maps and homotopies to account for discrepancies. If we think of least squares solutions as playing the role of the diagonal here, this note aims to repeat that story but now the \v{C}ech-theoretic cocycles are potentially homotopy-theoretic regression models. The conclusion is that locally defined least squares solutions $\mathfrak{a}_i$ can be encoded as linearized functions $\mathfrak{a}_i^T\cdot (a - \mathfrak{a}_i$ in Koszul degree zero, and their discrepancies on overlaps are witnessed by homotopies $\beta_{ij}$ in Koszul degree $1$. Of course, the discrepancies of these $\beta_{ij}$ on triple overlaps can then be witnessed by homotopies in degree $2$ and so on.

In the first section the presheaf of least squares Koszul complexes is constructed on the category of weighted finite data sets. The second section then deals with techniques of modding out by the local least squares solutions and offers a homotopy theoretic model in the form of a $0$-cocycle. The final section works out a toy example so that computability is more believable. 

\textbf{Computer Assistance.} In this project, as a first for me, I used ChatGPT-4 and ChatGPT-5 as a sounding board, to expedite some choices (ie presentation of constructions), and to save time checking routine examples. All proofs, computations, constructions, and phrasing that actually made it into this paper were my own.

\textbf{Acknowledgements.} I would like to thank Mahmoud Zeinalian for sharing the works of Toledo and Tong with me, as their constructions inspired everything in this paper. I would also like to thank Michael Robinson for encouraging conversations with regards to applications of sheaves of complexes.

\section{The least-squares Koszul presheaf}

The simplest case is used to setup some notation and definitions. Suppose an \emph{ambient space} $X\times Y = \{(x,y) \in \mathbb{R}^2\}$ is fixed to define a set $\Fin$ of finite \emph{data sets} $\mathscr{D}=\{(x^1, y^1), \ldots, (x^m, y^m)\}$, and the ``line of best fit'', $y = mx + b$, is the chosen model. Writing $\phi(x):= (x,1)^T$ and $a:= (m,b)^T$ then the space of parameters $\Lambda = \{ a= (b,m) \in \mathbb{R}^2\}$ and the above choice of \emph{model} is really a function (see for example \cite[chapters 1 and 6]{Va})
\begin{align*}
f: X \times \Lambda &\to Y = \mathbb{R}\\
(x, a) &\mapsto  \phi(x) \cdot a 
\end{align*}
which by definition is linear in $\Lambda$. Next, choosing the \emph{squared error loss function} $L(y, f(x,a)):= \left( y - f(x,a)\right)^2$ induces the \emph{squared error loss functional} $\Sigma {L}:\Fin X \times \Lambda \to Y$, which assigns to each $\mathscr{D} \in \Fin X$ and $a \in \Lambda$, the sum of all $L(y^i, f(x^i,a))$ over $(x^i, y^i) \in \mathscr{D}$. The goal in classical \emph{least squares (LS) regression} is to find $\mathfrak{a} \in \Lambda$ which minimizes $\Sigma {L}_{\mathscr{D}}:= \Sigma {L}(\mathscr{D}, -)$.

For a fixed data set $\mathscr{D}$, note then that when $f$ is linear in $\Lambda$, the induced functional $\Sigma {L}_{\mathscr{D}}: \Lambda \to Y$ is smooth so minimizing $\Sigma {L}_{\mathscr{D}}$ occurs when the gradient $\nabla \left(\Sigma {L}_{\mathscr{D}}\right)$ is zero. The induced equations for the components of this gradient, are referred to as the \emph{normal equations}. For example when $f$ is given by ``$y=mx+b$'' the normal equations are
\begin{align}
0&=\frac{\partial}{\partial m}\Sigma {L}_{\mathscr{D}} = \frac{\partial}{\partial m}\sum\limits_{i=1}^m \left( y^i - (mx^i + b)\right)^2= -2\sum\limits_{i=1}^m \left( y^i - (mx^i + b)\right)\cdot x^i \nonumber\\
\Leftrightarrow0&=  -2\sum\limits_{i=1}^m \left( y^i - (mx^i + b)\right)\cdot x^i \label{EQ LS normal m}\\
\text{and} \quad 0&=\frac{\partial}{\partial b}\Sigma {L}_{\mathscr{D}} = \frac{\partial}{\partial b}\sum\limits_{i=1}^m \left( y^i - (mx^i + b)\right)^2= -2\sum\limits_{i=1}^m \left( y^i - (mx^i + b)\right)\nonumber\\
\Leftrightarrow0&=  -2\sum\limits_{i=1}^m \left( y^i - (mx^i + b)\right).\label{EQ LS normal b}
\end{align}
When evaluated at a data set $\mathcal{D}$, these normal equations do not exhibit the right functoriality for later constructions. Instead, weights \cite[section 14.5.1]{Bishop} can be introduced to keep track of which points in the data set are being considered. Since the finite subsets of $X\times Y$ being used here are additionally ordered, they are equivalent to functions $(x,y):[n]=\{1, \ldots, n\}\to X\times Y$. So a \emph{weighted finite subset} will be a function $(\omega, x,y ): [n] \to X \times Y\times \mathbb{R}$ with \emph{weights} $\omega^i:= \omega(i)$. Fixing an ambient space $X\times Y:= \mathbb{R}^{N+1}$, the set of all weighted\footnote{It seems likely that the use of weights in the rest of the paper can instead be replaced by choices of density functions $\omega: X \to \mathbb{R}$, whereby whenever considering a subset you assign zero for all points outside the subset. However, weights felt a little more standard and the difference didn't feel meaningful.} finite subsets, letting the index set vary, will be written $\WFin$.  Now define a \emph{weighted loss functional} $\mathcal{L}: \WFin \times \Lambda \to \mathbb{R} $ which assigns to each weighted finite set $\omega\mathscr{D}= \{(\omega^i, x^i, y^i)\}_{i=1}^m \in \WFin$ and parameter $a \in \Lambda$, the sum of all $\omega^i \cdot L(y^i, f(x^i,a))$ over $(x^i, y^i) \in \mathscr{D}$. 

In the example of ``$y=mx+b$'' the non-weighted normal equations from \eqref{EQ LS normal m} and \eqref{EQ LS normal b}  become 
\begin{equation}\label{EQ weighted LS mx+b}
0= -2 \sum\limits_{i=1}^m \left( y^i - (mx^i + b)\right)\cdot  \omega^i x^i \quad \text{ and  } \quad  0=  -2\sum\limits_{i=1}^m \left( y^i - (mx^i + b)\right)\cdot \omega^i.
\end{equation}
Under reasonable assumptions the general form is quite similar (see lemma \ref{LEM eta linear}).

These normal equations will help produce a chain complex using the well-known Koszul complex \cite[chapter 17]{Ei} construction. One construction for a Koszul complex starts with choosing a ring $\mathcal{R}$, a rank $n$, and $n$-many elements $\eta^i \in \mathcal{R}$. In each degree $p=0, \ldots, n$ the $\mathcal{R}$-module $K_p(\mathcal{R})$ is defined as the exterior algebra of the $n$-dimensional free $\mathcal{R}$-module. In other words, writing the basis of this module as $\{ e^1, \ldots, e^n\}$,
\begin{equation}\label{EQ Koszul p}
K_p :=\left(  \bigoplus\limits_{i=1}^n \mathcal{R} \cdot  e^i \right)^{\otimes p} / \langle \left\{e^i \otimes e^i \vert i = 1, \ldots, n \right\}  \rangle
\end{equation}
with the quotient sending ``$\otimes$'' to ``$\wedge$''. The differential $\iota: K_{p+1} \to K_p$ is then defined via interior multiplication by $\mathcal{N}:= \sum\limits_{i=1}^n \eta^i \cdot e^i$  in $K_1(\mathcal{R})$,
\begin{equation}\label{EQ Koszul d}
\iota \left(e^{i_0} \cdots  e^{i_p}\right):= \sum\limits_{j=0}^p (-1)^j\eta^j \cdot e^{i_0} \cdots \widehat{e^{i_j}}  \cdots  e^{i_p}.
\end{equation}

For the rest of this section, fix: 
\begin{itemize} 
\item an \emph{ambient space} $X\times Y = \{ (x^1, \ldots, x^N,y) \in \mathbb{R}^{N+1}\}$, 
\item a \emph{parameter space} $\Lambda =\{a= (a^1, \ldots, a^n) \in \mathbb{R}^n\}$, 
\item a \emph{model} $f :X \times \Lambda \to Y = \mathbb{R}$, linear in $\Lambda$, 
\item the \emph{weighted LS loss functional} $\mathcal{L}: \WFin  \times \Lambda \to \mathbb{R}$, given by 
\begin{equation}
\mathcal{L}:\left(\omega\mathscr{D}, a \right) \mapsto \sum\limits_{i=1}^m (y^i- f(x^i, a))^2\cdot \omega^i,
\end{equation}
\item and, for each weighted finite data set $\omega\mathscr{D}$, the polynomial ring $\mathcal{R}^{\omega\mathscr{D}}:= \mathbb{R}[a, \omega]$ where the $\omega^i$ are \emph{weight variables} for data points in $\mathscr{D}$.
\end{itemize}

\begin{definition}\label{DEF big koszul}
For a finite data set $\mathscr{D}\subset \mathbb{R}^N$, the \emph{LS Koszul complex} $\left(K_{\bullet}(\mathcal{R}^{\omega\mathscr{D}}), \iota_{\omega\mathcal{D}} \right)$ is the complex defined above using the ring $\mathcal{R}^{\omega\mathscr{D}}$, the rank $n$, and $\eta^i:=\nabla\left(  \mathcal{L}_{\omega\mathscr{D}}\right)^{a^i}$ from \eqref{EQ gradient components weighted}.
\end{definition}
Linearity of the model with respect to parameters $a$ means that the normal equation components have a particularly useful form.
\begin{lemma}\label{LEM eta linear}
If $f(x, a) = \phi(x) \cdot a $ is linear in $\Lambda$ then $\eta^k= \nu^k(\omega) + N^k(\omega) \cdot a$ where $\nu^k$ and $N^{k, \bullet}$ are the scalar and vector-valued functions on $\omega$ respectively given by 
\begin{equation}\label{EQ gradient components weighted}
 \nu^k(\omega) := -2 \sum\limits_{j=1}^m y^j \cdot \phi(x^j)^{k} \omega^j \quad \text{and} \quad N^{k, \ell}(\omega):= -2 \sum\limits_{j=1}^m   \phi(x^j)^{k} \omega^j  \phi(x^j)^{\ell}.
\end{equation}
\begin{proof}
\begin{align*}
 \nabla  \left( \mathcal{L}_{\omega\mathscr{D}}\right)^{a^k} &= \frac{\partial }{\partial a^k} \sum\limits_{j=1}^m\left( y^j -   f(x^j, a^{\bullet})\right)^2\cdot \omega^j =\frac{\partial }{\partial a^k} \sum\limits_{j=1}^m\left( y^j -  \sum\limits_{\ell=1}^n \phi(x^j)^{\ell} \cdot a^{\ell} \right)^2\cdot \omega^j \\
& =-2 \sum\limits_{j=1}^m\left( y^j -  \sum\limits_{\ell=1}^n \phi(x^j)^{\ell} \cdot a^{\ell} \right)\cdot \phi(x^j)^{k} \omega^j \\
& =-2 \sum\limits_{j=1}^m y^j \cdot \phi(x^j)^{k} \omega^j -2 \sum\limits_{j=1}^m   \phi(x^j)^{k} \omega^j \sum\limits_{\ell=1}^n \phi(x^j)^{\ell} \cdot a^{\ell}\\ 
&= -2 \sum\limits_{j=1}^m y^j \cdot \phi(x^j)^{k} \omega^j -2\sum\limits_{\ell=1}^n \left(\sum\limits_{j=1}^m   \phi(x^j)^{k} \omega^j  \phi(x^j)^{\ell}\right) \cdot a^{\ell}. 
\end{align*}
\end{proof}
\end{lemma}

Weighted finite subsets form a category $\catWFin$, which forgets to the poset category on finite subsets. Objects in $\catWFin$ are weighted finite subsets $\omega\mathscr{D}= \omega \times x: [m] \to X \times \mathbb{R}$ and morphisms $\rho: \omega\mathscr{D} \to  \omega\mathscr{D}'$ are given by injections $\rho: [m] \hookrightarrow [m']$ satisfying $(\omega \times x) \circ \rho =\omega' \times x'$. The point of this construction is that the assignment of the polynomial ring $\mathcal{R}^{\omega\mathscr{D}}$ to each weighted data set $\omega\mathscr{D}$ forms a presheaf of rings after defining the pullback ring morphism
\begin{align}
\mathcal{R}^{\omega\mathscr{D}'} &\xrightarrow{\rho^*} \mathcal{R}^{\omega\mathscr{D}}\label{EQ ring map restrict}\\
a^i&\mapsto a^i \quad \text{ for } i = 1, \ldots, n\nonumber\\
\omega^{j'}&\mapsto \begin{cases}
\omega^j & \text{ if } j' = \rho(j)\\
0 & \text{ if } j' \notin \rho([m]).
\end{cases}\nonumber
\end{align}
For example given an inclusion of data sets $\{x^1, x^3\} \subset \{x^1, x^2, x^3, x^4\}$, then a polynomial $p(a^{\bullet}, \omega^1, \omega^2, \omega^3, \omega^4)$ will be sent to $p(a^{\bullet}, \omega^1,0, \omega^3 ,0)$ thought of as an element in $\mathbb{R}[a; \omega^1, \omega^2, \omega^3,\omega^4] / (\omega^2, \omega^4) \cong \mathbb{R}[a; \omega^1,\omega^3].$ 
\begin{theorem}
The assignment of the LS-Koszul complex $K_{\bullet}( \mathcal{R}^{\omega \mathscr{D}})$ to each weighted finite data set induces a presheaf of chain complexes $\mathcal{K}: \catWFin^{op} \to \Ch$, via the presheaf of rings described in \eqref{EQ ring map restrict}.
\begin{proof}
The Koszul complex construction is covariant with respect to ring morphisms, provided that the ring morphism preserves the differential. But the components of the differential \eqref{EQ gradient components weighted} are sums of quantities which vanish when $\omega^i$ is sent to zero. In this way the ring morphism \eqref{EQ ring map restrict} induces a chain map $\mathcal{K}(\rho):K_{\bullet}( \mathcal{R}^{\omega \mathscr{D}'})\to K_{\bullet}( \mathcal{R}^{\omega \mathscr{D}})$. 
\end{proof}
\end{theorem}

When the Koszul complex is constructed using a regular sequence $\eta^1, \ldots, \eta^n$ of ring elements, then the complex resolves the quotient ring $\mathcal{R}^{\omega \mathscr{D}}/(\eta^1, \ldots, \eta^n)$. In the case of $K_{\bullet}( \mathcal{R}^{\omega \mathscr{D}})$ then $H_0$ computes the coordinate ring generated by solutions which vanish on the normal equations. Of course, up to evaluating the weights at $\omega^i=1$, such solutions are precisely the least squares solutions. 

Let a \emph{cover} of a weighted finite data set $\omega\mathscr{D}$ be a collection $\mathcal{U} = \left\{ \omega U_i \xrightarrow{\rho_i} \omega\mathscr{D}\right\}_{i \in I}$ of morphisms so that the union of the images of the $x_i:[m_i] \to X$ equal $\mathscr{D}$. It is straightforward to evaluate the presheaf $\mathcal{K}$ on $\mathcal{U}$ to form a \v{C}ech-Koszul bicomplex $\Cech^{\bullet}\left(\mathcal{K}(\mathcal{U}) \right)$. However, its $0$-cocycles don't seem to capture the information I was looking for. Instead the following section builds a complex where the total $0$-cocycles offer a model for a homotopy theoretic regression model. That being said, it would not be surprising if someone more knowledgable found the above \v{C}ech-Koszul bicomplex to be sufficiently interesting to work with.

\section{Homotopy theoretic least squares solutions}
In this section, the rings $\mathcal{R}^{\omega \mathscr{D}}$ are linearized \cite[section 16]{Ei} near some least squares solution so that the induced \v{C}ech-Koszul total complex provides the kind of homotopic discrepancies between local least squares solutions that I am looking for. There are many choices in this section and I believe that each of them can be explored to provide different levels of control over these homotopies. The choice being made in this paper is to turn a least squares solution into an ideal $\mathfrak{I}$ of $\mathcal{R}^{\omega \mathscr{D}}$ and then work with the Koszul complex over $\mathcal{R}^{\omega \mathscr{D}}/\mathfrak{I}^2$. Just as is pointed out in \cite[section 9]{TT}, working$\mod \mathfrak{I}^2$ has the effect of strictifying lots of the data in the \v{C}ech direction. If instead this section worked modulo $\mathfrak{I}^3$ or higher, then higher terms could appear (especially if the normal equations for a given model are no longer affine in parameters $a$) and instead of constructing a total complex coming from a bi-complex, this paper would construct a total complex coming from a twisted complex. All of the constructions appear to still work, but for the sake of an initial clean construction with potentially immediate applications, the choice to work only with the first-order information is made. 

\begin{definition}
Given a weighted data set $\omega \mathscr{D}$ and a choice of parameters $\mathfrak{a} \in \Lambda$ define the ring 
\begin{equation}
\mathcal{R}^{\omega \mathscr{D}}_{\mathfrak{a}}:=\mathcal{R}^{\omega \mathscr{D}}/{\mathfrak{I}^2_\mathfrak{a}} .
\end{equation}
where $\mathfrak{I}_\mathfrak{a} \subset \mathcal{R}^{\omega \mathscr{D}}$ is the ideal generated by $(a^i - \mathfrak{a}^i)$ for $i=1, \ldots, n$.
\end{definition}
The canonical map $q_{\mathfrak{a}}: \mathcal{R}^{\omega \mathscr{D}}\to \mathcal{R}^{\omega \mathscr{D}}_{\mathfrak{a}}$ is given by 
\begin{equation}\label{EQ linearize polyn}
\left(q_{\mathfrak{a}}f\right)(a, \omega) \equiv f(\mathfrak{a}, \omega) + \sum\limits_{i=1}^n \frac{\partial f}{\partial a_i}(\mathfrak{a}, \omega) \cdot ( a_i - \mathfrak{a}_i) \mod \mathfrak{I}_{\mathfrak{a}}^2.
\end{equation}
Degree-wise, $q_{\mathfrak{a}}$ provides natural maps $\left(q_{\mathfrak{a}}\right)_*:K_{p}\left( \mathcal{R}^{\omega \mathscr{D}} \right) \to K_{p}\left( \mathcal{R}^{\omega \mathscr{D}}_{\mathfrak{a}} \right)$, and so applying \eqref{EQ linearize polyn} to the sequence $\left(\eta^i\right)_{i=1}^n$ given by the components of the partial\footnote{The gradient being used is only with respect to the $a^i$ variables, not the $\omega^i$ variables} gradient from definition \ref{DEF big koszul} yields,
\begin{align*}
\left(q_{\mathfrak{a}}(\eta)^i\right)(a, \omega)&=\eta^i(\mathfrak{a}, \omega) +   \sum\limits_{j=1}^n \frac{\partial \eta^i}{\partial a_j}(\mathfrak{a}, \omega) \cdot ( a_j - \mathfrak{a}_j) \mod \mathfrak{I}^2\\
\intertext{and by applying lemma \ref{LEM eta linear}}
&=\eta^i(\mathfrak{a}, \omega) +   N^i(\omega) \cdot ( a- \mathfrak{a}) \mod \mathfrak{I}^2.
\end{align*}
Note that $\mathfrak{a}$ is a LS solution then $\eta^i(\mathfrak{a}, \omega)  = 0$ and so  $\left(q_{\mathfrak{a}}(\eta)^i\right)(a, \omega)= N^i(\omega) \cdot ( a- \mathfrak{a}) \mod \mathfrak{I}^2$. 

\begin{definition}\label{DEF local Koszul}
For a finite data set $\mathscr{D}\subset \mathbb{R}^N$ and choice of least squares solution $\mathfrak{a}$, the \emph{LS-Koszul complex linearized near $\mathfrak{a}$}, written $\left(K_{\bullet}(\mathcal{R}^{\omega\mathscr{D}}_{\mathfrak{a}}), \iota_{\omega\mathcal{D}_{\mathfrak{a}}} \right)$, is the Koszul complex defined around \eqref{EQ Koszul p} using the ring $\mathcal{R}^{\omega\mathscr{D}}_{\mathfrak{a}}$, the rank $n$, and $\eta^i_{\mathfrak{a}}(a, \omega):=N^i(\omega) \cdot (a^{\bullet} - \mathfrak{a}^{\bullet})$.
\end{definition}
Note that this is also just a sort of partial Hessian of the least squares function. The above construction is such that the following chain maps are for free.
\begin{proposition}
The quotient map $q_{\mathfrak{a}}$ from linearizing $\mathcal{R}^{\omega\mathscr{D}}$ near a least squares solution $\mathfrak{a}$ induces a chain map of Koszul complexes $\left(q_{\mathfrak{a}}\right)_*: K_{\bullet}(\mathcal{R}^{\omega\mathscr{D}}) \to K_{\bullet}(\mathcal{R}^{\omega\mathscr{D}}_{\mathfrak{a}})$.
\end{proposition}

Different choices of least square solutions are not necessarily compatible with the restriction maps for the Koszul presheaf. In particular, consider a morphism $\omega U \xhookrightarrow{\rho} \omega V$ of finite weighted sets with least squares solution $\mathfrak{a}_V$ on  $\omega V$. The map $\rho^*$ from \ref{EQ ring map restrict} still applies on the level of rings, 
\begin{equation}
\mathcal{R}^{\omega V}_{\mathfrak{a}} \xrightarrow{\rho^*} \mathcal{R}^{\omega U}_{\mathfrak{a}},
\end{equation}
and so it applies on the level of graded modules over these rings. However, since $\mathfrak{a}$ is likely not a LS solution on $\omega U$, then the induced map will not be a chain map between the type of linearized Koszul complexes from definition \ref{DEF local Koszul}. Simply call the latter complex $\restr{K_{\bullet}(\mathcal{R}^{\omega\mathscr{D}}_{\mathfrak{a}})}{\omega U}$ with the understanding that the differential here is just given by applying $\rho^*$ to the components (ie ring elemnts) $\eta^i_{\mathfrak{a}}$, and then $\rho^*$ induces a chain map,
\begin{equation}\label{EQ restr local}
K_{\bullet}(\mathcal{R}^{\omega V}_{\mathfrak{a}}) \xrightarrow{\rho^*} \restr{K_{\bullet}(\mathcal{R}^{\omega V}_{\mathfrak{a}})}{\omega U}.
\end{equation}
Conveniently, these restrictions of linearized complexes are chain isomorphic to the chosen linearized complexes by translation as is now described.

On $\mathcal{R}^{\omega \mathscr{D}}$, define the assignment $\left(a, \omega\right) \xmapsto{\tau_{\mathfrak{a}, \mathfrak{b}}} \left(a- (\mathfrak{a} - \mathfrak{b}), \omega\right)$ which sends generators to polynomials so it induces a ring endomorphism. It follows that on generators of the ideal $\mathfrak{I}_\mathfrak{b}$,
\[ (a^i - \mathfrak{b}^i) \xmapsto{\tau_{\mathfrak{a}, \mathfrak{b}}}  \left(a^i- (\mathfrak{a}^i - \mathfrak{b}^i) \right) -  \mathfrak{b}^i= (a^i - \mathfrak{a}^i),\]
so that $\tau_{\mathfrak{a}, \mathfrak{b}}(\mathfrak{I}_\mathfrak{b}) \subset \mathfrak{I}_\mathfrak{a}$, and in particular $\tau_{\mathfrak{a}, \mathfrak{b}}(\mathfrak{I}^2_\mathfrak{b}) \subset \mathfrak{I}^2_\mathfrak{a}$. Since the inverse is $\tau_{\mathfrak{b}, \mathfrak{a}}$ then translation induces a ring isomorphism $\tau_{\mathfrak{a}, \mathfrak{b}}: \mathcal{R}^{\omega \mathscr{D}}_{\mathfrak{b}} \xrightarrow{\cong} \mathcal{R}^{\omega \mathscr{D}}_{\mathfrak{a}}$.
\begin{lemma}\label{LEM tau chain iso}
The translation map $\tau_{\mathfrak{a}, \mathfrak{b}}$ described above induces a chain isomorphism 
\begin{equation}\label{EQ tau chain iso}
\tau_{\mathfrak{b}, \mathfrak{a}}: K_{\bullet}(\mathcal{R}^{\omega \mathscr{D}}_{\mathfrak{a}})\xrightarrow{\cong} K_{\bullet}(\mathcal{R}^{\omega \mathscr{D}}_{\mathfrak{b}}).
\end{equation}
\begin{proof}
The discussion above the lemma shows how $\tau_{\mathfrak{b}, \mathfrak{a}}: \mathcal{R}^{\omega \mathscr{D}}_{\mathfrak{a}} \xrightarrow{\cong} \mathcal{R}^{\omega \mathscr{D}}_{\mathfrak{b}}$ is a ring isomorphism between the rings from which the complexes in \eqref{EQ tau chain iso} are built and so the map is an isomorphism of graded modules. All that is left is to show that the map commutes with differentials. It is sufficient to check that $\tau_{\mathfrak{b}, \mathfrak{a}}$ maps $\eta^i_{\mathfrak{a}}$ to $\eta^i_{\mathfrak{b}}$:
\begin{align*}
\tau_{\mathfrak{b}, \mathfrak{a}}\left(\eta^i_{\mathfrak{a}} (a, \omega)\right) &= \tau_{\mathfrak{b}, \mathfrak{a}}\left(N^i(\omega) \cdot (a^{\bullet} - \mathfrak{a}^{\bullet}) \right)= N^i(\omega) \cdot \tau_{\mathfrak{b}, \mathfrak{a}}(a^{\bullet} - \mathfrak{a}^{\bullet})\\
&= N^i(\omega) \cdot (a^{\bullet} - \mathfrak{b}^{\bullet})= \eta^i_{\mathfrak{b}} (a, \omega).
\end{align*}
\end{proof}
\end{lemma}  
The idea now is to cover $\omega \mathscr{D}$ by more than just weighted finite subsets, but additionally by (choices of) LS solutions on each (higher) intersection of covering sets. One can think about this as choosing ``local coordinates'' on each portion of the data set and such a construction is itself a simplicial presheaf. 
An approach similar to ours in \cite{GTZ} can be taken where these choices of parameters form their own simplicial presheaf, $\mathcal{A}: \catWFin^{op} \to \sSet$, where an $n$-simplex over a weighted data set $\omega \mathscr{D}$ is just\footnote{This construction is the $0$-coskeleton of the assignment of LS solutions to data sets.} $(n+1)$-many choices of model parameters $\overline{\mathfrak{a}}= (\mathfrak{a}_0, \ldots, \mathfrak{a}_n)$ over $\mathscr{D}$. A map $\rho: \omega U \hookrightarrow \omega V$ in $\catWFin$ applies the identity to the $n$-simplices of $\mathcal{A}$: $\rho^*(\overline{\mathfrak{a}}) = \overline{\mathfrak{a}}$. Then thinking of the nerve of chain complexes over the rings $\mathcal{R}^{\omega (-)}$ as another simplicial presheaf $\mathcal{N} \Ch: \catWFin^{op} \to \sSet$, the linearized Koszul complexes $K_{\bullet}(\mathcal{R}^{\omega\mathscr{D}}_{\mathfrak{a}})$ assemble into a natural transformation.
\begin{theorem}\label{THM K lin}
The assignment of each choice of LS solution $\mathfrak{a}$ over a weighted finite data set $\omega \mathscr{D}$ to the linearized Koszul complexes $K_{\bullet}(\mathcal{R}^{\omega\mathscr{D}}_{\mathfrak{a}})$, forms a map of simplicial presheaves $\mathcal{K}^{loc}: \mathcal{A} \to  \mathcal{N} \Ch$. 
\begin{proof}
First for each object $\omega\mathscr{D}$ an $n$-simplex in $\mathcal{A}(\omega\mathscr{D})$, given by $(n+1)$-many parameter choices $\overline{\mathfrak{a}}$ on $\mathscr{D}$ needs to be assigned to an $n$-simplex $\mathcal{K}^{lin}(\overline{\mathfrak{a}})$ in the nerve of chain complexes. The assignment is to send each vertex $\mathfrak{a}_i$ to the linearized Koszul complex $K_{\bullet}(\mathcal{R}^{\omega\mathscr{D}})_{\mathfrak{a}_i}$ and pairs of vertices $\mathfrak{a}_i\to \mathfrak{a}_j$ to the chain isomorphism from lemma \ref{LEM tau chain iso} given by translation. Next, $\mathcal{K}^{lin}$ needs to assign for each inclusion $\rho: \omega U \hookrightarrow \omega V$ a map of simplicial sets $\rho^*: \mathcal{K}^{lin}(\omega V) \to \mathcal{K}^{lin}(\omega U)$ which is essentially given by restriction \eqref{EQ restr local} and translation \eqref{EQ tau chain iso}. 
\end{proof} 
\end{theorem}

Covering a weighted finite data set $\omega \mathscr{D}$ with a collection of sub weighted finite data sets and choices of LS solutions $\mathcal{U}= \{ (\omega U_i, \mathfrak{a}_i)\}_{i \in I}$, there is a natural bicomplex by way of the presheaf discussion above. Evaluating the presheaf $\mathcal{K}^{lin}$ on the \v{C}ech nerve of the cover $\Cech \mathcal{U}$, produces a comsimplicial simplicial set whose totalization is a natural \emph{\v{C}ech-Koszul LS complex}. 

A $0$-cocycle in this complex (after some re-grading) would consist of:
\begin{itemize}
\item A choice of polynomial $p_i(a, \omega) \mod \mod \mathfrak{I}^2$ in $K_{0}\left( \mathcal{R}^{\omega U_i} \right)_{\mathfrak{a}_i}$ on each $\omega U_i$, 
\item A degree $1$ element $q_{ij} \in K_{1}\left( \mathcal{R}^{\omega U_{ij}} \right)_{\mathfrak{a}_{ij}}$ on each $\omega U_{ij}$ satisfying, 
\[ \iota q_{ij} =\tau_{\mathfrak{a}_{ij}, \mathfrak{a}_j} \left(  \restr{p_j}{U_{ij}}\right) -\tau_{\mathfrak{a}_{ij}, \mathfrak{a}_i} \left(  \restr{p_i}{U_{ij}}\right),\]
\item A degree $2$ element $r_{ijk} \in K_{2}\left( \mathcal{R}^{\omega U_{ijk}} \right)_{\mathfrak{a}_{ijk}}$ on each $\omega U_{ijk}$ whose image under the differential similarly witnesses higher disrepencies of $q_{ij}$'s, and so on. 
\end{itemize}
\section{A toy example}
This section extracts the core data from the previous description of a total degree $0$ \v{C}ech-Koszul cocycle and applies it to an actual data set with least squares solutions on local charts. The \v{C}ech-Koszul degree $(1,1)$ element $\beta_{12}$ constructed below is the homotopy between the first-order information provided by the local least squares solutions $\mathfrak{a}_{i}$.

Consider the data set 
\[\mathscr{D}=\{(x_i, y_i)\}_{i=1}^5=  \{ (-4, 2), (-1,1), (1,2), (2,4), (5,6)\}\]
and cover it by two sub sets $\mathscr{D}_1: = \{(x_i, y_i)\}_{i=1}^4$ and $\mathscr{D}_2: = \{(x_i, y_i)\}_{i=2}^5$ and label their intersection  $\mathscr{D}_{1,2}$. Using the model $y = mx + b$ so that using the notation above $a = (m,b)$ and $n=2$. 

Choose the LS solutions $\mathfrak{a}_1 = (\mathfrak{m}_1, \mathfrak{b}_1) = (11/42, 50/21)$, $\mathfrak{a}_2 = (\mathfrak{m}_2, \mathfrak{b}_2) = (13/15, 26/15)$, and $\mathfrak{a}_{1,2} = (\mathfrak{m}_{1,2}, \mathfrak{b}_{1,2}) = (13/14, 12/7)$. Note the differences $\mathfrak{a}_1 - \mathfrak{a}_{1,2} = (-2/3, 2/3)$ and $\mathfrak{a}_2 - \mathfrak{a}_{1,2} = (-13/210, 2/105)$.

Compute the weighted normal equations (\eqref{EQ LS normal m} and \eqref{EQ LS normal b}) on $\mathscr{D}_{1,2}$,
\begin{align*}
\eta^m_{12}&= -2 \sum\limits_{i=2}^4 \left( y^i - (mx^i + b)\right)\cdot x^i \omega^i\\
&= -2 \left( (1-(m\cdot (-1) + b)\cdot (-1) \omega^2 +(2-(m\cdot (1) + b)\cdot (1) \omega^3+ (4-(m\cdot (2) + b)\cdot (2) \omega^4 \right) \\
&= -2 \left( (-1-m + b)\omega^2 +(2-m- b) \omega^3+ (8-4m-2b) \omega^4 \right) \\
&= -2 \left( \left(- \omega^2 + 2 \omega^3 + 8 \omega^4 \right) + \left( -\omega^2- \omega^3-4 \omega^4\right)\cdot m + \left( \omega^2-\omega^3-2 \omega^4\right) \cdot b \right) \\
\text{while} \quad \eta^b_{12}&= -2 \sum\limits_{i=2}^4 \left( y^i - (mx^i + b)\right)\cdot \omega^i\\
&= -2 \left( (1-(m\cdot (-1) + b) \omega^2 +(2-(m\cdot (1) + b) \omega^3+ (4-(m\cdot (2) + b)\omega^4 \right) \\
&= -2 \left( (1+m- b) \omega^2 +(2-m- b) \omega^3+ (4-2m- b)\omega^4 \right) \\
&= -2 \left( \left(\omega^2 + 2 \omega^3 + 4\omega^4 \right) + \left( \omega^2- \omega^3-2 \omega^4\right)\cdot m + \left(- \omega^2-\omega^3- \omega^4\right) \cdot b \right) \\
\end{align*}
Recall from \eqref{EQ linearize polyn} and the ensuing discussion around $\eta^i$ linearized at a parameter choice, that the linear-in-parameters part is given by the Jacobian of $\eta$ (Hessian of the least squares functional), 
\begin{align}
 N&=  \begin{pmatrix} \partial \eta^m / \partial m & \partial \eta^m / \partial b \\ \partial \eta^b / \partial m & \partial \eta^b / \partial b\end{pmatrix}=- 2 \begin{pmatrix} -\omega^2- \omega^3-4 \omega^4 & \omega^2-\omega^3-2 \omega^4\\  \omega^2- \omega^3-2 \omega^4 & - \omega^2-\omega^3- \omega^4\end{pmatrix}\nonumber \\
 \intertext{and now plugging in weight values of $1$ for each $x_i$ in the intersection data set:}
N &= -2 \begin{pmatrix} -6 & -2\\ -2 & -3\end{pmatrix}= \begin{pmatrix} 12 & 4\\ 4 & 6\end{pmatrix}, \label{EQ toy N} \Rightarrow N^{-1}= \frac{1}{56}  \begin{pmatrix} 6 & -4\\ -4 & 12\end{pmatrix}.
 \end{align}

Thus in the localized setting, the components of the Koszul differential are given by the formula, 
\begin{equation}\label{EQ components eta toy}
\begin{pmatrix} \eta^m_{12} \\ \eta^b_{12} \end{pmatrix} = N \cdot \begin{pmatrix} m - \mathfrak{m}_{12} \\ b - \mathfrak{b}_{12} \end{pmatrix}  \mod \mathfrak{I}^2_{\mathfrak{a}} =\begin{pmatrix} 12 & 4\\ 4 & 6\end{pmatrix} \cdot \begin{pmatrix} m - \mathfrak{m}_{12} \\ b - \mathfrak{b}_{12} \end{pmatrix}  \mod \mathfrak{I}^2_{\mathfrak{a}}.
\end{equation}

The goal is to interpret the discrepancy between $\mathfrak{a}_1$ and $\mathfrak{a}_2$ as an element in $K_{0}\left( \mathcal{R}^{\omega \mathscr{D}_{12}} \right)_{\mathfrak{a}_{12}}$ and find $\beta_{12} \in K_{1}\left( \mathcal{R}^{\omega \mathscr{D}_{12}} \right)_{\mathfrak{a}_{12}}$ which ``witnesses the discrepancy'' in the sense that the differential applied to $\beta_{12}$ is that discrepancy. 

Let $\delta_{12}:= \mathfrak{a}_2 - \mathfrak{a}_1=(127/210,-68/105)  $, naively compute 
\begin{equation}
\begin{pmatrix} \beta_{12}^m \\  \beta_{12}^b \end{pmatrix}:= N^{-1}\delta_{12} = \begin{pmatrix}653/5880 \\  -1070/5880\end{pmatrix},
\end{equation}
and interpret this as the element $\beta_{12} =  \beta_{12}^m e^m +  \beta_{12}^b e^b \in K_{1}\left( \mathcal{R}^{\omega \mathscr{D}_{12}} \right)_{\mathfrak{a}_{12}}$. Applying the Koszul differential \eqref{EQ Koszul d} to this element, using the explicit components \eqref{EQ components eta toy} for this example yields
\begin{align*}
\iota\left(  \beta_{12}^m e^m +  \beta_{12}^b e^b\right) &=  \beta_{12}^m \cdot \eta_{12}^m +  \beta_{12}^b \eta_{12}^b\\
&= \beta_{12}^m \cdot \left( 12(m - \mathfrak{m}_{12}) +4 \cdot (b - \mathfrak{b}_{12})  \right)+ \beta_{12}^b \cdot \left( 4(m - \mathfrak{m}_{12}) +6 \cdot (b - \mathfrak{b}_{12})  \right)\\
&= \frac{127}{210} (m - \mathfrak{m}_{12}) - \frac{68}{105}(b - \mathfrak{b}_{12})= \delta_{12}^T \cdot \begin{pmatrix} m - \mathfrak{m}_{12} \\ b - \mathfrak{b}_{12} \end{pmatrix}= \delta^T \cdot \left( a - \mathfrak{a}_{12}\right) .
\end{align*}
From the above calculation, it becomes clearer how the higher terms in the Koszul complex track the discrepancies in the directions of the parameters using the normal equations. 

Letting $\alpha_{i}:= \mathfrak{a}_i^T \cdot (a - \mathfrak{a}_i)\in K_{0}\left( \mathcal{R}^{\omega \mathscr{D}_{i}}_{\mathfrak{a}_{i}}\right)$ for $i=1,2$ then $\Delta_{12}$ becomes the \v{C}ech differential applied to $\alpha_{\bullet}$. The sum $\alpha_{\bullet} + \beta_{12}$ is, in this toy example, a total degree $0$ cocycle in the \v{C}ech-Koszul total complex described after theorem \ref{THM K lin}.


\begin{thebibliography}{GMTZ99}
\bibitem[Ei]{Ei}Eisenbud, \emph{Title: Commutative Algebra: with a View Toward Algebraic Geometry}, Graduate Texts in Mathematics No 150, Springer-Verlag New York, Inc. 1995.

\bibitem[Va]{Va} V. Vapnik, \emph{The Nature of Statistical Learning Theory}, Springer, 2nd ed., 2000.

\bibitem[Bi]{Bishop} C.M. Bishop. \emph{Pattern Recognition and Machine Learning}. Springer, 2006.

\bibitem[GZ]{GZ} C. Glass and M. Zeinalian, \emph{Toledo-Tong \v{C}ech parametrix as a conjugacy invariant for subgroups of
biholomorphisms}, preprint.

\bibitem[GTZ]{GTZ} C. Glass, T. Tradler, and M. Zeinalian \emph{Hirzebruch-Riemann-Roch for complex analytic infinity-prestacks}, \url{https://arxiv.org/abs/2601.12454}. (2026)

\bibitem[TT]{TT} D. Toledo and Y.L. Tong. \emph{A parametrix for $\overline{\del}$ and Riemann-Roch in \v{C}ech theory}. Topology (1976) Vol. 15, 273-301.
\end{thebibliography}
\end{document}